\documentclass[11pt]{amsart}

\usepackage{epigamath}

\usepackage[english]{babel}

\numberwithin{equation}{section}

\usepackage[T1]{fontenc}
\usepackage[utf8]{inputenc}

\usepackage{amsmath,xypic,tikz-cd}						
\usepackage{amssymb}
\usepackage{amsthm}
\usepackage{amscd}
\usepackage{amsfonts}
\usepackage{stmaryrd}
\usepackage[all]{xy}
\usepackage{wasysym}
\usepackage{euler} 
 
\usepackage{extarrows}
\usepackage{enumerate}

\usepackage{color}

\usepackage{tikz}									
\usetikzlibrary{matrix}
\usetikzlibrary{patterns}
\usetikzlibrary{matrix}
\usetikzlibrary{positioning}
\usetikzlibrary{decorations.pathmorphing}

\newtheorem*{Theorem}{Main Theorem}
\newtheorem{theorem}{Theorem}[section]

\newtheorem{lemma}[theorem]{Lemma}

\theoremstyle{definition}

\newtheorem{definition}[theorem]{Definition}

\newtheorem{remark}[theorem]{Remark}

\DeclareMathOperator{\Eff}{Eff}

\newcommand{\DD}{\displaystyle}

\newcommand{\Moduli}{\overline{\mathcal{M}}}
\newcommand{\moduli}{\mathcal{M}}
\newcommand{\Hyp}{\overline{\mathcal{H}}}
\newcommand{\hyp}{\mathcal{H}}

\newcommand{\mathds}[1]{\mbox{\bf #1}}

\EpigaVolumeYear{4}{2020} \EpigaArticleNr{2} \ReceivedOn{October 19,
2018}
\InFinalFormOn{September 11, 2019}
\AcceptedOn{January 20, 2020}

\title{Hyperelliptic classes are rigid and extremal in genus two}
\titlemark{Hyperelliptic classes are rigid and extremal in genus two}

\author{Vance Blankers}
\address{Department of Mathematics, Northeastern University, Boston, Massachusetts 02115-5005}
\email{blankersv@northeastern.edu}

\authormark{V. Blankers}

\AbstractInEnglish{We show that the class of the locus of hyperelliptic curves with $\ell$ marked Weierstrass points, $m$ marked conjugate pairs of points, and $n$ free marked points is rigid and extremal in the cone of effective codimension-($\ell+m$) classes on $\Moduli_{2,\ell+2m+n}$. This generalizes work of Chen and Tarasca and establishes an infinite family of rigid and extremal classes in arbitrarily-high codimension.}



\TitleInFrench{Les classes hyperelliptiques sont rigides et extr\'emales en genre 2}

\AbstractInFrench{Nous montrons que la classe du lieu des courbes hyperelliptiques avec $\ell$ points de Weierstrass marqu\'es, $m$ paires de points conjugu\'es marqu\'es et $n$ points marqu\'es libres est rigide et extr\'emale dans le c\^one des classes effectives de codimension $\ell+m$ de $\Moduli_{2,\ell+2m+n}$. Ceci g\'en\'eralise le travail de Chen et Tarasca et construit une famille infinie de classes rigides et extr\'emales en codimension arbitrairement grande.}

\acknowledgement{The author was supported by NSF FRG grant 1159964 (PI: Renzo Cavalieri).}

\begin{document}



\maketitle

\begin{prelims}

\DisplayAbstractInEnglish

\bigskip

\DisplayKeyWords

\medskip

\DisplayMSCclass

\bigskip

\languagesection{Fran\c{c}ais}

\bigskip

\DisplayTitleInFrench

\medskip

\DisplayAbstractInFrench

\end{prelims}


\newpage

\setcounter{tocdepth}{2}

\tableofcontents


\section*{Introduction}
\label{sec-intro}

Every smooth curve of genus two admits a unique degree-two \emph{hyperelliptic} map to $\mathbb{P}^1$. The Riemann-Hurwitz formula forces such a map to have six ramification points called \emph{Weierstrass points}; each non-Weierstrass point $p$ exists as part of a \emph{conjugate pair} $(p,p')$ such that the images of $p$ and $p'$ agree under the hyperelliptic map.

The locus of curves of genus two with $\ell$ marked Weierstrass points is codimension $\ell$ inside the moduli space $\moduli_{2,\ell}$, and in \cite{chentarasca} it is shown that the class of the closure of this locus is rigid and extremal in the cone of effective classes of codimension $\ell$. Our main theorem extends their result to $\hyp_{2,\ell,2m,n}\subseteq \moduli_{2,\ell+2m+n}$, 
the locus of genus-two curves with $\ell$ marked Weierstrass points, $m$ marked conjugate pairs, and $n$ free marked points (see Definition \ref{def-hyp}).

\begin{Theorem}
For $\ell,m,n\geq 0$, the class $\Hyp_{2,\ell,2m,n}$, if non-empty, is rigid and extremal in the cone of effective classes of codimension $\ell+m$ in $\Moduli_{2,\ell+2m+n}$.
\end{Theorem}

In \cite{chencoskun2015}, the authors show that the effective cone of codimension-two classes of $\Moduli_{2,n}$ has infinitely many extremal cycles for every $n$. Here we pursue a perpendicular conclusion: although in genus two $\ell \leq 6$, the number of conjugate pairs and number of free marked points are unbounded, so that the classes $\Hyp_{2,\ell,2m,n}$ form an infinite family of rigid and extremal cycles in arbitrarily-high codimension. Moreover, the induction technique used to prove the main result is genus-agnostic, pointing towards a natural extension of the main theorem to higher genus given a small handful of low-codimension cases.

When $\ell + m \geq 3$, our induction argument (Theorem \ref{thm-main}) is a generalization of that used in \cite[Theorem 4]{chentarasca} to include conjugate pairs and free points; it relies on pushing forward an effective decomposition of one hyperelliptic class onto other hyperelliptic classes and showing that the only term of the decomposition to survive all pushforwards is the original class itself. This process is straightforward when there are at least three codimension-one conditions available to forget; however, when $\ell + m = 2$, and in particular when $\ell = 2$ and $m = 0$, more care must be taken. The technique used in \cite[Theorem 5]{chentarasca} to overcome this problematic subcase relies on an explicit expression for $\left[\Hyp_{2,2,0,0}\right]$ which becomes cumbersome when a non-zero number of free marked points are allowed. Although adding free marked points can be described via pullback, pullback does not preserve rigidity and extremality in general, so we introduce an intersection-theoretic calculation using tautological $\omega$-classes to handle this case instead.

The base case of the induction (Theorem \ref{thm-base}) is shown via a criterion (Lemma \ref{lem-divisor}) given by  \cite{chencoskun2014} for rigidity and extremality for divisors; it amounts to an additional pair of intersection calculations. We utilize the theory of moduli spaces of admissible covers to construct a suitable curve class for the latter intersection, a technique which generalizes that used in \cite{rulla01} for the class of $\Hyp_{2,1,0,0}$.

\subsection*{Structure of the paper.} We begin in \textsection \ref{sec-prelim} with some background on $\Moduli_{g,n}$ and cones of effective cycles. This section also contains the important Lemma \ref{lem-divisor} upon which Theorem \ref{thm-base} depends. In \textsection \ref{sec-main}, we prove Theorem \ref{thm-base}, which establishes the base case for the induction argument of our main result, Theorem \ref{thm-main}. Finally, we conclude in \textsection \ref{sec-highergenus} with a discussion of extending these techniques for $g\geq 3$ and possible connections to a CohFT-like structure.

\subsection*{Acknowledgments.} The author wishes to thank Nicola Tarasca, who was kind enough to review an early version of the proof of the main theorem and offer his advice. The author is also greatly indebted to Renzo Cavalieri for his direction and support.

\section{Preliminaries on $\Moduli_{g,n}$ and effective cycles}
\label{sec-prelim}

\subsection*{Moduli spaces of curves, hyperelliptic curves, and admissible covers.} We work throughout in $\Moduli_{g,n}$, the moduli space of isomorphism classes of stable genus $g$ curves with $n$ (ordered) marked points. If $2g-2+n > 0$ this space is a smooth Deligne-Mumford stack of dimension $3g-3+n$. We denote points of $\Moduli_{g,n}$ by $[C; p_1,\dots,p_n]$ with $p_1,\dots,p_n\in C$ smooth marked points. For fixed $g$, we may vary $n$ to obtain a family of moduli spaces related by \emph{forgetful morphisms}: for each $1\leq i \leq n$, the map $\pi_{p_i}:\Moduli_{g,n}\to\Moduli_{g,n-1}$ forgets the $i$th marked point and stabilizes the curve if necessary. The maps $\rho_{p_i}:\Moduli_{g,n} \to \Moduli_{g,\{p_i\}}$ are the \emph{rememberful morphisms} which are the composition of all possible forgetful morphisms other than $\pi_{p_i}$. 

Due to the complexity of the full Chow ring of $\Moduli_{g,n}$, the \emph{tautological ring} $R^*(\Moduli_{g,n})$ is often considered instead \cite{faberpandharipande} (for both rings we assume rational coefficients). Among other classes, this ring contains the classes of the boundary strata, as well as all $\psi$- and $\lambda$-classes. For $1\leq i \leq n$ the class $\psi_{p_i}$ is defined to be the first Chern class of the line bundle on $\Moduli_{g,n}$ whose fiber over a given isomorphism class of curves is the cotangent line bundle at the $i$th marked point of the curve; $\lambda_1$ is the first Chern class of the Hodge bundle. The tautological ring also includes pullbacks of all $\psi$- and $\lambda$-classes, including the $\omega$\emph{-classes}, sometimes called \emph{stable} $\psi$\emph{-classes}. The class $\omega_{p_i}$ is defined on $\Moduli_{g,n}$ for $g,n\geq 1$ as the pullback of $\psi_{p_i}$ along $\rho_{p_i}$. Several other notable cycles are known to be tautological, including the hyperelliptic classes considered below (\cite{faberpandharipande}).

\emph{Hyperelliptic curves} are those which admit a degree-two map to $\mathbb{P}^1$. The Riemann-Hurwitz formula implies that a hyperelliptic curve of genus $g$ contains $2g+2$ Weierstrass points which ramify over the branch locus in $\mathbb{P}^1$. For a fixed genus, specifying the branch locus allows one to recover the complex structure of the hyperelliptic curve and hence the hyperelliptic map. Thus for $g\geq 2$, the codimension of the locus of hyperelliptic curves in $\Moduli_{g,n}$ is $g-2$. In this context, requiring that a marked point be Weierstrass (resp. two marked points be a conjugate pair) is a codimension-one condition for genus at least two.

We briefly use the theory of \emph{moduli spaces of admissible covers} to construct a curve in $\Moduli_{2,n}$ in Theorem \ref{thm-base}. These spaces are particularly nice compactifications of Hurwitz schemes. For a thorough introduction, the standard references are \cite{harrismumford} and \cite{acv}. For a more hands-on approach in the same vein as our usage, see as well \cite{cavalierihodgeint}.

\subsection*{Notation.} We use the following notation for boundary strata on $\Moduli_{g,n}$
; all cycles classes are given as stack fundamental classes. For $g\geq 1$, the divisor class of the closure of the locus of irreducible nodal curves is denoted by $\delta_{irr}$. By $\delta_{h,P}$ we mean the class of the divisor whose general element has one component of genus $h$ attached to another component of genus $g-h$, with marked points $P$ on the genus $h$ component and marked points $\{p_1,\dots,p_n\}\backslash P$ on the other. By convention $\delta_{h,P} = 0$ for unstable choices of $h$ and $P$.

Restrict now to the case of $g=2$. We use $W_{2,P}$ to denote the codimension-two class of the stratum whose general element agrees with that of $\delta_{2,P}$, with the additional requirement that the node be a Weierstrass point. We denote by $\gamma_{1,P}$ the class of the closure of the locus of curves whose general element has a genus-one component containing the marked points $P$ meeting in two points conjugate under a hyperelliptic map a rational component with marked points $\{p_1,\dots,p_n\}\backslash P$ (see Figure \ref{wfigure}).

\begin{figure}[t]
\begin{tikzpicture}


\draw[very thick] (0,0) ellipse (4cm and 1.5cm);

\draw[very thick] (5,0) circle (1);
\node at (3.7,0) {$w$};

\draw[very thick] (-2.3,-0.15) .. controls (-2.1,0.2) and (-0.9,0.2) .. (-0.7,-0.15); 
\draw[very thick] (-2.5,0) .. controls (-2.2,-0.4) and (-0.8,-0.4) .. (-0.5,0); 

\draw[very thick] (2.3,-0.15) .. controls (2.1,0.2) and (0.9,0.2) .. (0.7,-0.15); 
\draw[very thick] (2.5,0) .. controls (2.2,-0.4) and (0.8,-0.4) .. (0.5,0); 

\fill (-1.8,0.7) circle (0.10); \node at (-1.4,0.7) {$p_1$};
\fill (-1.4,-0.7) circle (0.10); \node at (-1,-0.7) {$p_2$};
\fill (1.1,-0.7) circle (0.10); \node at (1.5,-0.7) {$p_3$};
\fill (5.05,0.5) circle (0.10); \node at (5.45,0.5) {$p_4$};
\fill (5.05,-0.5) circle (0.10); \node at (5.45,-0.5) {$p_5$};


\draw[very thick] (0,-4.5) ellipse (4cm and 1.5cm);

\draw[very thick] (3.2,-3.6) .. controls (5.4,-3.8) and (5.4,-5.2) .. (3.2,-5.4); 
\draw[very thick] (3.2,-3.6) .. controls (7.0,-2.6) and (7.0,-6.4) .. (3.2,-5.4);
\node at (3.2,-3.4) {$+$}; 
\node at (3.2,-5.8) {$-$}; 

\draw[very thick] (-0.8,-4.65) .. controls (-0.6,-4.3) and (0.6,-4.3) .. (0.8,-4.65); 
\draw[very thick] (-1,-4.5) .. controls (-0.7,-4.9) and (0.7,-4.9) .. (1,-4.5); 

\fill (-1.8,-3.8) circle (0.10); \node at (-1.4,-3.8) {$p_1$};
\fill (-1.4,-5.2) circle (0.10); \node at (-1,-5.2) {$p_2$};
\fill (1.1,-5.2) circle (0.10); \node at (1.5,-5.2) {$p_3$};
\fill (5.05,-4) circle (0.10); \node at (5.45,-4) {$p_4$};
\fill (5.05,-5) circle (0.10); \node at (5.45,-5) {$p_5$};


\draw[very thick]  (9,0) circle (0.30);
\node at (9,0) {$2$};
\draw[very thick] (8.79,0.21) -- (8.5,0.5); \node at (8.2,0.5) {$p_1$};
\draw[very thick] (8.7,0) -- (8.3,0); \node at (8,0) {$p_2$};
\draw[very thick] (8.79,-0.21) -- (8.5,-0.5); \node at (8.2,-0.5) {$p_3$};
\draw[very thick] (9.30,0) -- (10.5,0);

\fill (10.5,0) circle (0.10);
\draw[very thick] (10.5,0) -- (11,0.5); \node at (11.3,0.5) {$p_4$};
\draw[very thick] (10.5,0) -- (11,-0.5); \node at (11.4,-0.5) {$p_5$};

\node at (9.8,0.2) {$w$};


\draw[very thick]  (9,-4.5) circle (0.30);
\node at (9,-4.5) {$1$};
\draw[very thick] (8.79,-4.29) -- (8.5,-4); \node at (8.2,-4) {$p_1$};
\draw[very thick] (8.7,-4.5) -- (8.3,-4.5); \node at (8,-4.5) {$p_2$};
\draw[very thick] (8.79,-4.71) -- (8.5,-5); \node at (8.2,-5) {$p_3$};

\draw[very thick] (9.2,-4.3) .. controls (9.6,-3.9) and (10.2,-3.9) .. (10.5,-4.5); 
\draw[very thick] (9.2,-4.7) .. controls (9.6,-5.1) and (10.2,-5.1) .. (10.5,-4.5);

\fill (10.5,-4.5) circle (0.10);
\draw[very thick] (10.5,-4.5) -- (11,-4); \node at (11.3,-4) {$p_4$};
\draw[very thick] (10.5,-4.5) -- (11,-5); \node at (11.4,-5) {$p_5$};

\node at (9.8,-3.8) {$+$};
\node at (9.8,-5.2) {$-$};

\end{tikzpicture}
\caption{On the left-hand side, the topological pictures of the general elements of $W_{2,P}$ (top) and $\gamma_{1,P}$ (bottom) in $\overline{\mathcal{M}}_{2,5}$ with $P = \{p_1,p_2,p_3\}$. On the right-hand side, the corresponding dual graphs.}
\label{wfigure}
\end{figure}
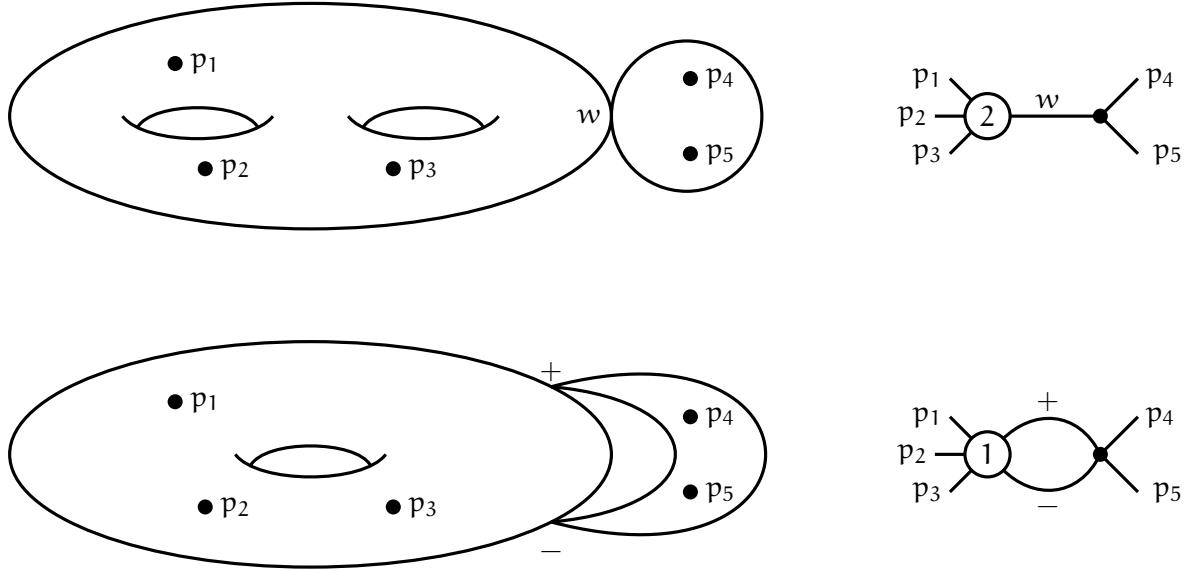

The space $\overline{Adm}_{2\xrightarrow{2} 0,t_1,\dots,t_6,u_{1\pm},\dots,u_{n\pm}}$ is the moduli space of degree-two admissible covers of genus two with marked ramification points (Weierstrass points) $t_i$ and marked pairs of points (conjugate pairs) $u_{j+}$ and $u_{j-}$. This space comes with a finite map $c$ to $\Moduli_{0,\{t_1,\dots,t_6,u_{1},\dots,u_{n}\}}$ which forgets the cover and remembers only the base curve and its marked points, which are the images of the markings on the source. It comes also with a degree $2^n$ map $s$ to $\Moduli_{2,1+n}$ which forgets the base curve and all $u_{j+}$ and $t_i$ other than $t_1$ and remembers the (stabilization of the) cover.

\begin{figure}[!ht]
\begin{tikzpicture}

\draw[very thick]  (0,0) circle (0.30);
\node at (0,0) {$1$};
\draw[very thick] (-0.21,0.21) -- (-0.5,0.5); \node at (-0.8,0.5) {$t_1$};
\draw[very thick] (-0.30,0) -- (-0.7,0); \node at (-1,0) {$t_2$};
\draw[very thick] (-0.21,-0.21) -- (-0.5,-0.5); \node at (-0.8,-0.5) {$t_6$};
\draw[very thick] (0.30,0) -- (1.5,0);

\fill (1.5,0) circle (0.10);
\draw[very thick] (1.5,0) -- (1,0.5); \node at (0.8,0.6) {$t_3$};

\draw[very thick] (3.5,0) .. controls (3.2,0.9) and (1.8,0.9) .. (1.5,0); 
\draw[very thick] (3.5,0) .. controls (3.2,-0.9) and (1.8,-0.9) .. (1.5,0);

\fill (3.5,0) circle (0.10);
\draw[very thick] (3.5,0) -- (4,0.5); \node at (4.2,0.6) {$t_4$};
\draw[very thick] (3.5,0) -- (4.7,0);

\fill (4.7,0) circle (0.10);
\draw[very thick] (4.7,0) -- (5.2,0.5); \node at (5.5,0.5) {$t_5$};
\draw[very thick] (4.7,0) -- (5.4,0); \node at (5.8,0) {$u_{1+}$};
\draw[very thick] (4.7,0) -- (5.2,-0.5); \node at (5.6,-0.5) {$u_{1-}$};

\end{tikzpicture}

\begin{tikzpicture}

\draw[very thick] (0,0) -- (0,-1.5);
\draw[very thick] (0,-1.5) -- (-0.15,-1.3);
\draw[very thick] (0,-1.5) -- (0.15,-1.3);

\end{tikzpicture}

\begin{tikzpicture}

\fill (.3,0) circle (0.10);
\draw[very thick] (.3,0) -- (-0.2,0.5); \node at (-0.5,0.5) {$t_1$};
\draw[very thick] (.3,0) -- (-0.4,0); \node at (-.7,0) {$t_2$};
\draw[very thick] (.3,0) -- (-0.2,-0.5); \node at (-0.5,-0.5) {$t_6$};
\draw[very thick] (.3,0) -- (1.5,0);

\fill (1.5,0) circle (0.10);
\draw[very thick] (1.5,0) -- (1.5,0.5); \node at (1.5,0.7) {$t_3$};
\draw[very thick] (1.5,0) -- (3.5,0);

\fill (3.5,0) circle (0.10);
\draw[very thick] (3.5,0) -- (3.5,0.5); \node at (3.5,0.7) {$t_4$};
\draw[very thick] (3.5,0) -- (4.7,0);

\fill (4.7,0) circle (0.10);
\draw[very thick] (4.7,0) -- (5.2,0.5); \node at (5.5,0.5) {$t_5$};
\draw[very thick] (4.7,0) -- (5.2,-0.5); \node at (5.6,-0.5) {$u_{1}$};

\end{tikzpicture}
\caption{An admissible cover in $\overline{Adm}_{2\xrightarrow{2} 0,t_1,\dots,t_6,u_{1\pm}}$ represented via dual graphs. In degree two the topological type of the cover is uniquely recoverable from the dual graph presentation.}
\label{admfigure}
\end{figure}
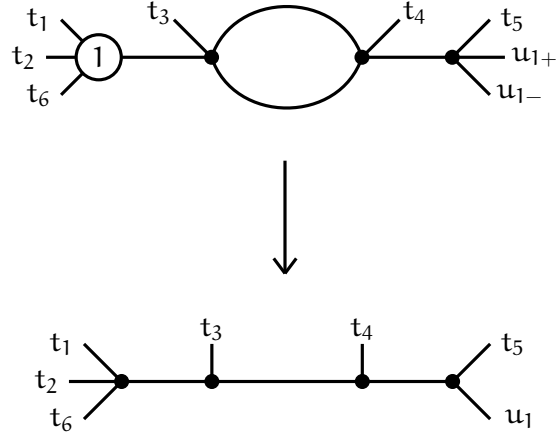

\subsection*{$\omega$-class lemmas.} The following two lemmas concerning basic properties of $\omega$-classes prove useful in the last subcase of Theorem \ref{thm-main}. The first is a unique feature of these classes, and the second is the $\omega$-class version of the dilaton equation.

\begin{lemma}
\label{lem-fallomega}
Let $g\geq 1$, $n\geq 2$, and $P\subset \{p_1,\dots,p_n\}$ such that $|P|\leq n-2$. Then for any $p_i,p_j\not\in P$
\begin{align*}
\omega_{p_i}\cdot \delta_{g,P} = \omega_{p_j}\cdot \delta_{g,P}
\end{align*}
on $\Moduli_{g,n}$.
\end{lemma}
\begin{proof}
This follows immediately from Lemma 1.9 in \cite{bcomega}.
\end{proof}

\begin{lemma}
\label{lem-omegadil}
Let $g,n\geq 2$. Then on $\Moduli_{g,n}$,
\begin{align*}
\pi_{p_i*}\omega_{p_j} = 2g - 2
\end{align*}
if $i=j$, and $0$ otherwise.
\end{lemma}
\begin{proof}
Let $P = \{p_1,\dots,p_n\}$. When $i=j$, the pushforward reduces to the usual dilaton equation for $\psi_{p_i}$ on $\Moduli_{g,\{p_i\}}$. If $\pi$ is the morphism which forgets all marked points, the diagram
\begin{center}
\begin{tikzcd}[row sep=1cm, column sep=1cm] 
\Moduli_{g,n} \arrow[d, "\rho_{p_i}"] \arrow[r, "\pi_{p_i}"] & \Moduli_{g,P\backslash\{p_i\}} \arrow[d, "\pi"] \\
\Moduli_{g,\{p_i\}} \arrow[r, "\pi_{p_i}"] & \Moduli_{g}
\end{tikzcd}
\end{center}
commutes, so $\pi_{p_i*}\omega_{p_i} = \pi_{p_i*}\rho_{p_i}^*\psi_{p_i} = \pi^*\pi_{p_i*}\psi_{p_i} = (2g-2)\mathds{1}$.

If $i\neq j$, then $\pi_{p_i*}\omega_{p_j} = \pi_{p_i*}\pi_{p_i}^*\omega_{p_j} = 0$.
\end{proof}

\subsection*{Cones and properties of effective classes.} For a projective variety $X$, the sum of two effective codimension-$d$ classes is again effective, as is any $\mathbb{Q}_+$-multiple of the same. This gives a natural convex cone structure on the set of effective classes of codimension $d$ inside the $\mathbb{Q}$ vector space of all codimension-$d$ classes, called the \emph{effective cone of codimension-$d$ classes} and denoted $\Eff^d(X)$. Given an effective class $E$ in the Chow ring of $X$, an \emph{effective decomposition of} $E$ is an equality
\begin{align*}
E = \sum_{s=1}^m a_sE_s
\end{align*}
with $a_s > 0$ and $E_s$ irreducible effective cycles on $X$ for all $s$. The main properties we are interested in for classes in the pseudo-effective cone are rigidity and extremality.

\begin{definition}
\label{def-rigex}
Let $E\in\Eff^d(X)$.

$E$ is \emph{rigid} if any effective cycle with class $rE$ is supported on the support of $E$.

$E$ is \emph{extremal} if, for any effective decomposition of $E$, all $E_s$ are proportional to $E$.
\end{definition}

When $d=1$, elements of the cone correspond to divisor classes, and the study of $\Eff^1(\Moduli_{g,n})$ is fundamental in the theory of the birational geometry of these moduli spaces. For example, $\Moduli_{0,n}$ is known to fail to be a Mori dream space for $n\geq 10$ (first for $n\geq 134$ in \cite{castravettevelev}, then for $n\geq 13$ in \cite{gonzalezkaru}, and the most recent bound in \cite{hkl2016}). For $n\geq 3$ in genus one, \cite{chencoskun2014} show that $\Moduli_{1,n}$ is not a Mori dream space; the same statement is true for $\Moduli_{2,n}$ by \cite{mullane2017}. In these and select other cases, the pseudo-effective cone of divisors has been shown to have infinitely many extremal cycles and thus is not rational polyhedral (\cite{chencoskun2015}).

These results are possible due in large part to the following lemma, which plays an important role in Theorem \ref{thm-base}. Here a \emph{moving curve $\mathcal{C}$ in $D$} is an irreducible effective curve $\mathcal{C}$, the deformations of which cover a Zariski-dense subset of $D$.

\begin{lemma}[{{\cite[Lemma 4.1]{chencoskun2014}}}]
\label{lem-divisor}
Let $D$ be an irreducible effective divisor in a projective variety $X$, and suppose that $\mathcal{C}$ is a moving curve in $D$ satisfying $\DD \int_{X}[D]\cdot[\mathcal{C}] < 0$. Then $[D]$ is rigid and extremal. \hfill $\square$
\end{lemma} 

\begin{remark}
Using Lemma \ref{lem-divisor} to show a divisor $D$ is rigid and extremal in fact shows more: if the lemma is satisfied, the boundary of the pseudo-effective cone is polyhedral at $D$. We do not rely on this fact, but see \cite[\textsection 6]{opie2016} for further discussion.
\end{remark}

Lemma \ref{lem-divisor} allows us to change a question about the pseudo-effective cone into one of intersection theory and provides a powerful tool in the study of divisor classes. Unfortunately, it fails to generalize to higher-codimension classes, where entirely different techniques are needed. Consequently, much less is known about $\Eff^d(\Moduli_{g,n})$ for $d\geq 2$. This paper is in part inspired by \cite{chentarasca}, where the authors show that certain hyperelliptic classes of higher codimension are rigid and extremal in genus two. In \cite{chencoskun2015}, the authors develop additional extremality criteria to show that in codimension-two there are infinitely many extremal cycles in $\Moduli_{1,n}$ for all $n\geq 5$ and in $\Moduli_{2,n}$ for all $n\geq 2$, as well as showing that two additional hyperelliptic classes of higher genus are extremal. These criteria cannot be used directly for the hyperelliptic classes we consider; this is illustrative of the difficulty of proving rigidity and extremality results for classes of codimension greater than one.

\section{Main theorem}
\label{sec-main}

In this section we prove our main result, which culminates in Theorem \ref{thm-main}. The proof proceeds via induction, with the base cases given in Theorem \ref{thm-base}. We begin by defining hyperelliptic classes on $\Moduli_{g,n}$.
\begin{definition}
\label{def-hyp}
Fix integers $\ell,m,n\geq 0$. Denote by $\Hyp_{g,\ell,2m,n}$ the closure of the locus of hyperelliptic curves in $\Moduli_{g,\ell+2m+n}$ with marked Weierstrass points $w_1,\dots,w_\ell$; pairs of marked points $+_1,-_1,\dots,+_m,-_m$ with $+_j$ and $-_j$ conjugate under the hyperelliptic map; and \emph{free} marked points $p_1,\dots,p_n$ with no additional constraints. By \emph{hyperelliptic class}, we mean a non-empty class equivalent to some $\left[\Hyp_{g,\ell,2m,n}\right]$ in the Chow ring of $\Moduli_{g,\ell+2m+n}$.
\end{definition}

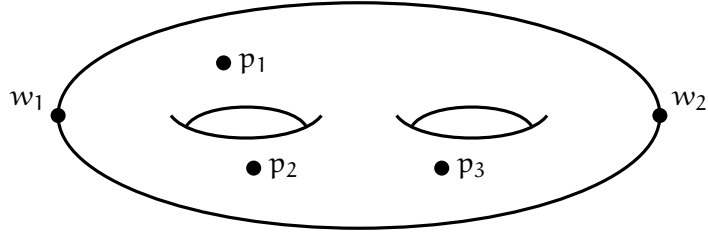
\begin{figure}[ht]
\begin{tikzpicture}

\draw[very thick] (0,0) ellipse (4cm and 1.5cm);

\draw[very thick] (-2.3,-0.15) .. controls (-2.1,0.2) and (-0.9,0.2) .. (-0.7,-0.15); 
\draw[very thick] (-2.5,0) .. controls (-2.2,-0.4) and (-0.8,-0.4) .. (-0.5,0); 

\draw[very thick] (2.3,-0.15) .. controls (2.1,0.2) and (0.9,0.2) .. (0.7,-0.15); 
\draw[very thick] (2.5,0) .. controls (2.2,-0.4) and (0.8,-0.4) .. (0.5,0); 

\fill (-4,0) circle (0.10); \node at (-4.4,0.2) {$w_1$};
\fill (4,0) circle (0.10); \node at (4.4,0.2) {$w_2$};
\fill (-1.8,0.7) circle (0.10); \node at (-1.4,0.7) {$p_1$};
\fill (-1.4,-0.7) circle (0.10); \node at (-1,-0.7) {$p_2$};
\fill (1.1,-0.7) circle (0.10); \node at (1.5,-0.7) {$p_3$};

\end{tikzpicture}
\caption{The general element of $\Hyp_{2,2,0,3}$.}
\label{hfigure}
\end{figure}

Lemma \ref{lem-divisor} allows us to establish the rigidity and extremality of the two divisor hyperelliptic classes for genus two, which together provide the base case for Theorem \ref{thm-main}.

\begin{theorem}
\label{thm-base}
For $n \geq 0$, the class of $\Hyp_{2,0,2,n}$ is rigid and extremal in $\Eff^1(\Moduli_{2,2+n})$ and the class of $\Hyp_{2,1,0,n}$ is rigid and extremal in $
\Eff^1(\Moduli_{2,1+n})$.
\end{theorem}

\begin{proof}
Define a moving curve $\mathcal{C}$ in $\Hyp_{2,0,2,n}$ by fixing a general genus-two curve $C$ with $n$ free marked points $p_1,\dots,p_n$ and varying the conjugate pair $(+,-)$. 

Since $\left[\Hyp_{2,0,2,n}\right] = \pi_{p_n}^*\cdots\pi_{p_1}^*\left[\Hyp_{2,0,2,0}\right]$, by the projection formula and the identity (see \cite{logan2003}) 
\begin{align*}
\left[\Hyp_{2,0,2,0}\right] = -\lambda + \psi_+ + \psi_- - 3\delta_{2,\varnothing} - \delta_{1,\varnothing},
\end{align*}
we compute
\begin{align*}
\int_{\Moduli_{2,2+n}} \left[\Hyp_{2,0,2,n}\right]\cdot [\mathcal{C}] &= \int_{\Moduli_{2,2}} \left[\Hyp_{2,0,2,0}\right] \cdot \pi_{p_1*}\cdots\pi_{p_n*}[\mathcal{C}] \\
&= 0 + (4-2+6) + (4-2+6) - 3(6) - 0 \\
&= -2.
\end{align*}
In particular, intersecting with $\lambda$ is $0$ by projection formula. Intersecting with either $\psi$-class can be seen as follows: pullback $\psi_i$ from $\Moduli_{2,1}$ to $\psi_i - \delta_{2,\varnothing}$, then use projection formula on $\psi_i$ back to $\Moduli_{2,1}$. This is just $2g-2$, since $\psi_i$ is the first Chern class of the cotangent bundle of $C$ over $i$. The intersection with $\delta_{2,\varnothing}$ corresponds to the $2g+2$ Weierstrass points. Finally, $\delta_{1,\varnothing}$ intersects trivially, since by fixing $C$ we have only allowed rational tail degenerations. As $\Hyp_{2,0,2,n}$ is irreducible, it is rigid and extremal by Lemma \ref{lem-divisor}.

We next apply Lemma \ref{lem-divisor} by constructing a moving curve $\mathcal{B}$ which intersects negatively with $\Hyp_{2,1,0,n}$ using the following diagram. Note that the image of $s$ is precisely $\Hyp_{2,1,0,n} \subset \Moduli_{2,1+n}$.

\begin{center}
\begin{tikzcd}[row sep=1cm, column sep=1cm] 
\overline{Adm}_{2\xrightarrow{2} 0,t_1,\dots,t_6,u_{1\pm},\dots,u_{n\pm}} \arrow[d, "c"] \arrow[r, "s"] & \Moduli_{2,1+n} \\
\Moduli_{0,\{t_1,\dots,t_6,u_{1},\dots,u_{n}\}} \arrow[d, "\pi_{t_6}"] \\
\Moduli_{0,\{t_1,\dots,t_5,u_{1},\dots,u_{n}\}} 
\end{tikzcd}
\end{center}

Fix a generic point $[x_n]$ in $\moduli_{0,\{t_1,\dots,t_5,u_{1},\dots,u_{n}\}}$ corresponding to a smooth marked curve and the point $[b_n]$ in $\Moduli_{0,\{t_1,\dots,t_5,u_{1},\dots,u_{n}\}}$ corresponding to a chain of $\mathbb{P}^1$s with $n+3$ components and marked points as shown in Figure \ref{pfigure} (if $n=0$, $t_4$ and $t_5$ are on the final component; if $n=1$, $t_5$ and $u_1$ are on the final component; etc.), and define $[\mathcal{X}_n] = s_*c^*\pi_{t_6}^*[x_n]$ and $[\mathcal{B}_n] = s_*c^*\pi_{t_6}^*[b_n]$ (with an additional relabeling of $t_1$ to $w_1$ and $u_{j_-}$ to $p_j$). Now $\mathcal{X}_n$ is a moving curve in $\Hyp_{2,1,0,n}$, and the deformations of $\mathcal{X}_n$ are parametrized by $\Moduli_{0,\{t_1,\dots,t_5,u_{1},\dots,u_{n}\}}$.


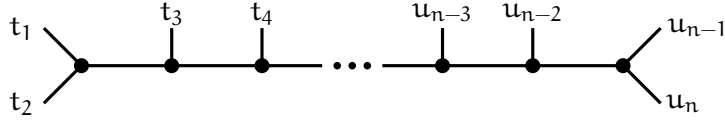
\begin{figure}[ht]
\begin{tikzpicture}

\fill (0,0) circle (0.10);
\draw[very thick] (0,0) -- (-0.5,0.5); \node at (-0.8,0.5) {$t_1$};
\draw[very thick] (0,0) -- (-0.5,-0.5); \node at (-0.8,-0.5) {$t_2$};
\draw[very thick] (0,0) -- (1.2,0);

\fill (1.2,0) circle (0.10);
\draw[very thick] (1.2,0) -- (1.2,0.5); \node at (1.2,0.7) {$t_3$};
\draw[very thick] (1.2,0) -- (2.4,0);

\fill (2.4,0) circle (0.10);
\draw[very thick] (2.4,0) -- (2.4,0.5); \node at (2.4,0.7) {$t_4$};
\draw[very thick] (2.4,0) -- (3.2,0);
\fill (3.4,0) circle (0.05); \fill (3.6,0) circle (0.05); \fill (3.8,0) circle (0.05);

\draw[very thick] (4,0) -- (4.8,0);
\fill (4.8,0) circle (0.10);
\draw[very thick] (4.8,0) -- (4.8,0.5); \node at (4.8,0.7) {$u_{n-3}$};
\draw[very thick] (4.8,0) -- (6,0);

\fill (6,0) circle (0.10);
\draw[very thick] (6,0) -- (6,0.5); \node at (6,0.7) {$u_{n-2}$};
\draw[very thick] (6,0) -- (7.2,0);

\fill (7.2,0) circle (0.10);
\draw[very thick] (7.2,0) -- (7.7,0.5); \node at (8.2,0.5) {$u_{n-1}$};
\draw[very thick] (7.2,0) -- (7.7,-0.5); \node at (8,-0.5) {$u_{n}$};

\end{tikzpicture}
\caption{The point $[b_n]$ in $\overline{\mathcal{M}}_{0,\{t_1,\dots,t_5,u_{1},\dots,u_{n}\}}$.}
\label{pfigure}
\end{figure}
Because the image of $s$ is $\Hyp_{2,1,0,n}$, the intersection $\left[\Hyp_{2,1,0,n}\right] \cdot [\mathcal{X}_n]$ is not transverse, so we correct with minus the Euler class of the normal bundle of $\Hyp_{2,1,0,n}$ in $\Moduli_{2,1+n}$ restricted to $\mathcal{X}_n$. Further, as all points in $\Moduli_{0,\{t_1,\dots,t_5,u_{1},\dots,u_{n}\}}$ are equivalent, we may replace $[\mathcal{X}_n]$ with $[\mathcal{B}_n]$ in the intersection. In other words,
\begin{align*}
\int_{\Moduli_{2,1+n}} \left[\Hyp_{2,1,0,n}\right] \cdot [\mathcal{X}_n] &= \int_{\Moduli_{2,1+n}} -\pi_{p_n}^*\cdots\pi_{p_1}^*\psi_{w_1}\cdot[\mathcal{X}_n] \\
&= \int_{\Moduli_{2,1+n}} -\pi_{p_n}^*\cdots\pi_{p_1}^*\psi_{w_1}\cdot[\mathcal{B}_n] \\
&= \int_{\Moduli_{2,1}} -\psi_{w_1}\cdot[\mathcal{B}_{0}].
\end{align*}
By passing to the space of admissible covers (see, for example, \cite{rulla01}), this integral is seen to be a positive multiple (a power of two) of
\begin{align*}
\int_{\Moduli_{1,2}} -\psi_{w_1}\cdot \left[\Hyp_{1,2,0,0}\right] &= \int_{\Moduli_{1,2}} -\psi_{w_1}\cdot (3\psi_{w_1}) \\
&= -\frac{1}{8},
\end{align*}
where we have used the fact that $\left[\Hyp_{1,2,0,0}\right] = 3\psi_{w_1}$ \cite{cavalierihurwitz}. Therefore, by Lemma \ref{lem-divisor}, $\Hyp_{2,1,0,n}$ is rigid and extremal.
\end{proof}

This establishes the base case for the inductive hypothesis in Theorem \ref{thm-main}. The induction procedure differs fundamentally for the codimension-two classes, so we first prove the following short lemma to simplify the most complicated of those.

\begin{lemma}
\label{lem-notprop}
The class $W_{2,\{p_1,\dots,p_n\}}$ is not proportional to $\left[\Hyp_{2,2,0,n}\right]$ on $\Moduli_{2,2+n}$.
\end{lemma}
\begin{proof}
Let $P = \{p_1,\dots,p_n\}$. Note that in $W_{2,P}$ the marked points $w_1$ and $w_2$ carry no special restrictions, and the class is of codimension two. Because the point $w_1$ is on a three-pointed rational component of the general element of $W_{2,P}$,
\begin{align*}
\int_{\Moduli_{2,2+n}} W_{2,P} \cdot \psi_{w_1}^{n+3} &= \int_{\Moduli_{0,3}} \psi_{w_1}^{n+3} = 0.
\end{align*}
However, using the equality
\begin{align*}
\left[\Hyp_{2,2,0,0}\right] = 6\psi_{w_1}\psi_{w_2} - \frac{3}{2}(\psi_{w_1}^2+\psi_{w_2}^2) - (\psi_{w_1} + \psi_{w_2})\left(\frac{21}{10}\delta_{1,\{w_1\}} + \frac{3}{5}\delta_{1,\varnothing} + \frac{1}{20}\delta_{irr}\right)
\end{align*}
established in \cite[Equation 4]{chentarasca} and Faber's Maple program \cite{faberprogram}, we compute
\begin{align*}
\int_{\Moduli_{2,2+n}} \left[\Hyp_{2,2,0,n}\right] \cdot \psi_{w_1}^{n+3} &= \int_{\Moduli_{2,2+n}} \pi_{p_1}^*\cdots\pi_{p_n}^*\left[\Hyp_{2,2,0,0}\right] \cdot \psi_{w_1}^{n+3} \\
&= \int_{\Moduli_{2,2}} \left[\Hyp_{2,2,0,0}\right] \cdot \pi_{p_1*}\cdots\pi_{p_n*}\psi_{w_1}^{n+3} \\
&= \int_{\Moduli_{2,2}} \Bigg(6\psi_{w_1}\psi_{w_2} - \frac{3}{2}(\psi_{w_1}^2+\psi_{w_2}^2) \\
& \hspace{1.5cm} - (\psi_{w_1} + \psi_{w_2})\left(\frac{21}{10}\delta_{1,\{w_1\}} + \frac{3}{5}\delta_{1,\varnothing} + \frac{1}{20}\delta_{irr}\right)\Bigg) \cdot \psi_{w_1}^3 \\
&= \frac{1}{384},
\end{align*}
so $W_{2,P}$ is not a non-zero multiple of $\left[\Hyp_{2,2,0,n}\right]$.
\end{proof}

We are now ready to prove our main result. The bulk of the effort is in establishing extremality, though the induction process does require rigidity at every step as well. Although we do not include it until the end, the reader is free to interpret the rigidity argument as being applied at each step of the induction.

The overall strategy of the extremality portion of the proof is as follows. Suppose $\left[\Hyp_{2,\ell,2m,n}\right]$ is given an effective decomposition. We show (first for the classes of codimension at least three, then for those of codimension two) that any terms of this decomposition which survive pushforward by $\pi_{w_i}$ or $\pi_{+_j}$ must be proportional to the hyperelliptic class itself. Therefore we may write $\left[\Hyp_{2,\ell,2m,n}\right]$ as an effective decomposition using only classes which vanish under pushforward by the forgetful morphisms; this is a contradiction, since the hyperelliptic class itself survives pushforward.

\begin{theorem}
\label{thm-main}
For $\ell,m,n\geq 0$, the class $\Hyp_{2,\ell,2m,n}$, if non-empty, is rigid and extremal in $\Eff^{\ell+m}(\Moduli_{2,\ell+2m+n})$.
\end{theorem}
\allowdisplaybreaks

\begin{proof}
We induct on codimension; assume the claim holds when the class is codimension $\ell+m-1$. Theorem \ref{thm-base} is the base case, so we may further assume $\ell+m \geq 2$. Now, suppose that
\begin{align}
\label{eq-decomp}
\left[\Hyp_{2,\ell,2m,n}\right] = \sum_{s} a_s\left[X_s\right] + \sum_{t} b_t\left[Y_t\right]
\end{align}
is an effective decomposition with $\left[X_s\right]$ and $\left[Y_t\right]$ irreducible codimension-$(\ell+m)$ effective cycles on $\Moduli_{2,\ell+2m+n}$, with $[X_s]$ surviving pushforward by some $\pi_{w_i}$ or $\pi_{+_j}$ and $[Y_t]$ vanishing under all such pushforwards, for each $s$ and $t$.

Fix an $\left[X_s\right]$ appearing in the right-hand side of $\eqref{eq-decomp}$. If $\ell\neq 0$, suppose without loss of generality (on the $w_i$) that $\pi_{w_1*}\left[X_s\right] \neq 0$. Since 
\begin{align*}
\pi_{w_1*}\left[\Hyp_{2,\ell,2m,n}\right] = (6-(\ell-1))\left[\Hyp_{2,\ell-1,2m,n}\right]
\end{align*}
is rigid and extremal by hypothesis, $\pi_{w_1*}\left[X_s\right]$ is a positive multiple of the class of $\Hyp_{2,\ell-1,2m,n}$ and $X_s\subseteq (\pi_{w_1})^{-1}\Hyp_{2,\ell-1,2m,n}$. By the commutativity of the following diagrams and the observation that hyperelliptic classes survive pushforward by all $\pi_{w_i}$ and $\pi_{+_j}$, we have that $\pi_{w_i*}\left[X_s\right] \neq 0$ and $\pi_{+_j*}\left[X_s\right] \neq 0$ for all $i$ and $j$.


\begin{center}
\begin{tikzcd}[row sep=1cm, column sep=1cm] 
\Hyp_{2,\ell,2m,n} \arrow[d, "\pi_{w_1}"] \arrow[r, "\pi_{+_j}"] & \Hyp_{2,\ell,2(m-1),n+1} \arrow[d, "\pi_{w_1}"] & & \Hyp_{2,\ell,2m,n} \arrow[d, "\pi_{w_1}"] \arrow[r, "\pi_{w_i}"] & \Hyp_{2,\ell-1,2m,n} \arrow[d, "\pi_{w_1}"] \\
\Hyp_{2,\ell-1,2m,n} \arrow[r, "\pi_{+_j}"] & \Hyp_{2,\ell-1,2(m-1),n+1} & & \Hyp_{2,\ell-1,2m,n} \arrow[r, "\pi_{w_i}"] & \Hyp_{2,\ell-2,2m,n}
\end{tikzcd}
\end{center}

If $\ell = 0$, suppose without loss of generality (on the $+_j$) that $\pi_{+_1*}\left[X_s\right] \neq 0$. Then the same conclusion holds that $\left[X_s\right]$ survives all pushforwards by $\pi_{+_j}$, since
\begin{align*}
\pi_{+_1*}\left[\Hyp_{2,\ell,2m,n}\right] = \left[\Hyp_{2,\ell,2(m-1),n+1}\right]
\end{align*}
is rigid and extremal by hypothesis, and $\pi_{+_1}$ commutes with $\pi_{+_j}$.

It follows that for any $\ell+m\geq 2$
\begin{align*}
\DD X_s \subseteq \bigcap_{i,j}\left((\pi_{w_i})^{-1}\Hyp_{2,\ell-1,2m,n}\cap (\pi_{+_j})^{-1}\Hyp_{2,\ell,2(m-1),n+1}\right).
\end{align*}
We now have two cases. If $\ell+m \geq 3$, any $\ell + 2m - 1$ non-free marked points in a general element of $X_s$ are distinct Weierstrass or conjugate pair marked points, and hence all $\ell + 2m$ such non-free marked points in a general element of $X_s$ are distinct Weierstrass or conjugate pair marked points. In this case we conclude that $\left[X_s\right]$ is a positive multiple of $\left[\Hyp_{2,\ell,2m,n}\right]$. If $\ell+m = 2$, we must analyze three subcases.


If $\ell = 0$ and $m = 2$, then
\begin{align*}
\DD X_s\subseteq (\pi_{+_1})^{-1}\Hyp_{2,0,2,n+1} \cap (\pi_{+_2})^{-1}\Hyp_{2,0,2,n+1}.
\end{align*}
The modular interpretation of the intersection leaves three candidates for $\left[X_s\right]$: $W_{2,P}$ or $\gamma_{1,P}$ for some $P$ containing neither conjugate pair, or $\left[\Hyp_{2,0,4,n}\right]$ itself. However, for the former two, we have $\dim W_{2,P} \neq \dim \pi_{+_1}(W_{2,P})$ and $\dim \gamma_{1,P} \neq \dim \pi_{+_1}(\gamma_{1,P})$ for all such $P$, contradicting our assumption that the class survived pushforward. Thus $\left[X_s\right]$ is proportional to $\left[\Hyp_{2,0,4,n}\right]$.

If $\ell = 1$ and $m = 1$, similar to the previous case, $\left[X_s\right]$ could be $\left[\Hyp_{2,1,2,n}\right]$ or $W_{2,P}$ or $\gamma_{1,P}$ for some $P$ containing neither the conjugate pair nor the Weierstrass point. However, if $X_s$ is either of the latter cases, we have $\dim X_s \neq \dim \pi_{+_1}(X_s)$, again contradicting our assumption about the non-vanishing of the pushforward, and so again $[X_s]$ must be proportional to $\left[\Hyp_{2,1,2,n}\right]$.

If $\ell = 2$ and $m = 0$, as before, $\left[X_s\right]$ is either $\left[\Hyp_{2,2,0,n}\right]$ itself or $W_{2,P}$ or $\gamma_{1,P}$ for $P = \{p_1,\dots,p_n\}$. Now $\dim W_{2,P} = \dim \pi_{w_i}W_{2,P},$ so the argument given in the other subcases fails (though $\gamma_{1,P}$ is still ruled out as before). Nevertheless, we claim that $W_{2,P}$ cannot appear on the right-hand side of (\ref{eq-decomp}) for $\Hyp_{2,2,0,n}$; to show this we induct on the number of free marked points $n$. The base case of $n=0$ is established in \cite[Theorem 5]{chentarasca}, so assume that $\Hyp_{2,2,0,n-1}$ is rigid and extremal for some $n\geq 1$. Suppose for the sake of contradiction that
\begin{align}
\label{eq-decomp2}
\left[\Hyp_{2,2,0,n}\right] = a_0W_{2,P} + \sum_{s} a_s\left[Z_s\right]
\end{align}
is an effective decomposition with each $\left[Z_s\right]$ an irreducible codimension-two effective cycle on $\Moduli_{2,2+n}$. Note that
\begin{align*}
W_{2,P} = \pi_{p_n}^*W_{2,P\backslash\{p_n\}} - W_{2,P\backslash\{p_n\}}.
\end{align*}
Multiply (\ref{eq-decomp2}) by  $\omega_{p_n}$ and push forward by $\pi_{p_n}$. On the left-hand side,
\begin{align*}
\pi_{p_n*}\left(\omega_{p_n}\cdot\left[\Hyp_{2,2,0,n}\right]\right) &= \pi_{p_n*}\left(\omega_{p_n}\cdot \pi_{p_n}^*\left[\Hyp_{2,2,0,n-1}\right]\right) \\
&= \pi_{p_n*} \left(\omega_{p_n}\right) \cdot \left[\Hyp_{2,2,0,n-1}\right] \\
&= 2 \left[\Hyp_{2,2,0,n-1}\right],
\end{align*}
having applied Lemma \ref{lem-omegadil}. Combining this with the right-hand side,
\begin{align*}
2\left[\Hyp_{2,2,0,n-1}\right] &= a_0\pi_{p_n*}\left(\omega_{p_n}\cdot \pi_{p_n}^*W_{2,P\backslash\{p_n\}} - \omega_{p_n}\cdot W_{2,P\backslash\{p_n\}}\right) + \sum_{s} a_s\pi_{p_n*}\left(\omega_{p_n}\cdot \left[Z_s\right]\right) \\
&= 2a_0W_{2,P\backslash\{p_n\}} + \pi_{p_n*}\left(\omega_{p_n}\cdot W_{2,P\backslash\{p_n\}}\right) + \sum_{s} a_s\pi_{p_n*}\left(\omega_{p_n}\cdot \left[Z_s\right]\right).
\end{align*}
The term $\pi_{p_n*}\left(\omega_{p_n}\cdot W_{2,P\backslash\{p_n\}}\right)$ vanishes by Lemma \ref{lem-fallomega}:
\begin{align*}
\pi_{p_n*}\left(\omega_{p_n}\cdot W_{2,P\backslash\{p_{n-1}\}}\right) &= \pi_{p_n*} \left(\omega_{w_1}\cdot W_{2,P\backslash\{p_n\}}\right) \\
&= \pi_{p_n*} \left(\pi_{p_n}^*\omega_{w_1}\cdot W_{2,P\backslash\{p_n\}}\right) \\
&= \omega_{w_1}\cdot \pi_{p_n*} W_{2,P\backslash\{p_n\}} \\
&= 0,
\end{align*}
where $w_1$ is the Weierstrass singular point on the genus-two component of $W_{2,P\backslash\{p_n\}}$. Altogether, we have
\begin{align*}
2\left[\Hyp_{2,2,0,n-1}\right] &= 2a_0W_{2,P\backslash\{p_n\}} + \sum_{s} a_s\pi_{p_n*}\left(\omega_{p_n}\cdot \left[Z_s\right]\right).
\end{align*}
\cite{rulla01} establishes that $\psi_{p_n}$ is semi-ample on $\Moduli_{2,\{p_n\}}$, so $\omega_{p_n}$ is semi-ample, and hence this is an effective decomposition. By hypothesis, $\Hyp_{2,2,0,n-1}$ is rigid and extremal, so $W_{2,P\backslash\{p_n\}}$ must be a non-zero multiple of $\left[\Hyp_{2,2,0,n-1}\right]$, which contradicts Lemma \ref{lem-notprop}. Therefore $W_{2,P}$ cannot appear as an $\left[X_s\right]$ in (\ref{eq-decomp}).

Thus for all cases of $\ell+m = 2$ (and hence for all $\ell+m\geq 2$), we conclude that each $\left[X_s\right]$ in (\ref{eq-decomp}) is a positive multiple of $\left[\Hyp_{2,\ell,2m,n}\right]$. Now subtract these $\left[X_s\right]$ from \eqref{eq-decomp} and rescale, so that
\begin{align*}
\left[\Hyp_{2,\ell,2m,n}\right] = \sum_{t} b_t\left[Y_t\right].
\end{align*}
Recall that each $\left[Y_t\right]$ is required to vanish under all $\pi_{w_i*}$ and $\pi_{+_j*}$. But the pushforward of $\left[\Hyp_{2,\ell,2m,n}\right]$ by any of these morphisms is non-zero, so there are no $[Y_t]$ in \eqref{eq-decomp}. Hence $\left[\Hyp_{2,\ell,2m,n}\right]$ is extremal in $\Eff^{\ell+m}(\Moduli_{2,\ell+2m+n})$.

For rigidity, suppose that $E:= r\left[\Hyp_{2,\ell,2m,n}\right]$ is effective. Since $\pi_{w_i*}E = (6-(\ell-1))r\left[\Hyp_{2,\ell-1,2m,n}\right]$ and $\pi_{+_j*}E = r\left[\Hyp_{2,\ell,2(m-1),n+1}\right]$ are rigid and extremal for all $i$ and $j$, we have that $\pi_{w_i*}E$ is supported on $\Hyp_{2,\ell-1,2m,n}$ and $\pi_{+_j*}E$ is supported on $\Hyp_{2,\ell,2(m-1),n+1}$. This implies that $E$ is supported on the intersection of $(\pi_{w_i})^{-1}\left[\Hyp_{2,\ell-1,2m,n}\right]$ and $(\pi_{+_j})^{-1}\left[\Hyp_{2,\ell,2(m-1),n+1}\right]$ for all $i$ and $j$. Thus $E$ is supported on $\Hyp_{2,\ell,2m,n}$, so $\left[\Hyp_{2,\ell,2m,n}\right]$ is rigid.
\end{proof}

\section{Higher genus}
\label{sec-highergenus}

The general form of the inductive argument in Theorem \ref{thm-main} holds independent of genus for $g\geq 2$. However, for genus greater than one, the locus of hyperelliptic curves in $\moduli_{g}$ is of codimension $g-2$, so that the base cases increase in codimension as $g$ increases.  The challenge in showing the veracity of the claim for hyperelliptic classes in arbitrary genus is therefore wrapped up in establishing the base cases of codimension $g-1$ (corresponding to Theorem \ref{thm-base}) and codimension $g$ (corresponding to the three $\ell+m = 2$ subcases in Theorem \ref{thm-main}).

In particular, our proof of Theorem \ref{thm-base} relies on the fact that $\Hyp_{2,0,2,n}$ and $\Hyp_{2,1,0,n}$ are divisors, and the subcase $\ell = 2$ in Theorem \ref{thm-main} depends on our ability to prove Lemma \ref{lem-notprop}. This in turn requires the description of $\Hyp_{2,2,0,0}$ given by \cite{chentarasca}. More subtly, we require that $\psi_{p_n}$ be semi-ample in $\Moduli_{2,\{p_n\}}$, which is known to be false in genus greater than two in characteristic 0 \cite{keel99}. In genus three, \cite{chencoskun2015} show that the base case $\Hyp_{3,1,0,0}$ is rigid and extremal, though it is unclear if their method will extend to $\Hyp_{3,1,0,n}$. Moreover, little work has been done to establish the case of a single conjugate pair in genus three, and as the cycles move farther from divisorial classes, such analysis becomes increasingly more difficult.

One potential avenue to overcome these difficulties is suggested by work of Renzo Cavalieri and Nicola Tarasca \cite{cavalieritarasca2017}. They use an inductive process to describe hyperelliptic classes in terms of decorated graphs using the usual dual graph description of the tautological ring of $\Moduli_{g,n}$. Such a formula for the three necessary base cases would allow for greatly simplified intersection-theoretic calculations, similar to those used in Theorem \ref{thm-base} and Lemma \ref{lem-notprop}. Though such a result would be insufficient to completely generalize our main theorem, it would be a promising start.

We also believe the observation that pushing forward and pulling back by forgetful morphisms moves hyperelliptic classes to (multiples of) hyperelliptic classes is a useful one. There is evidence that a more explicit connection between marked Weierstrass points, marked conjugate pairs, and the usual gluing morphisms between moduli spaces of marked curves exists as well, though concrete statements require a better understanding of higher genus hyperelliptic loci. Although it is known that hyperelliptic classes do not form a cohomological field theory over the full $\Moduli_{g,n}$, a deeper study of the relationship between these classes and the natural morphisms among the moduli spaces may indicate a CohFT-like structure, which in turn would shed light on graph formulas or other additional properties.


\providecommand{\bysame}{\leavevmode\hbox to3em{\hrulefill}\thinspace}


\begin{thebibliography}{X}

\bibitem[ACV01]{acv}
Dan Abramovich, Alessio Corti, and Angelo Vistoli, \emph{Twisted bundles and
  admissible covers}, Comm. in Algebra \textbf{31} (2001), no.~8, 3547--3618.

\bibitem[BC18]{bcomega}
Vance Blankers and Renzo Cavalieri, \emph{Intersections of $\omega$ classes in
  {$\Moduli_{g,n}$}}, Proceedings of G\"{o}kova Geometry--Topology Conference
  2017 (2018), 37--52.

\bibitem[Cav06]{cavalierihodgeint}
Renzo Cavalieri, \emph{Hodge-type integrals on moduli spaces of admissible
  covers}, Geom. Topol. Monogr. \textbf{8} (2006), 167--194.

\bibitem[Cav16]{cavalierihurwitz}
\bysame, \emph{Hurwitz theory and the double ramification cycle}, Jpn. J. Math.
  \textbf{11} (2016), no.~2, 305--331.

\bibitem[CC14]{chencoskun2014}
Dawei Chen and Izzet Coskun, \emph{Extremal effective divisors on
  {$\Moduli_{1,n}$}}, Math. Ann. \textbf{359} (2014), no.~3-4, 891--908.

\bibitem[CC15]{chencoskun2015}
\bysame, \emph{Extremal higher codimension cycles on moduli spaces of curves},
  Proc. London Math. Soc. \textbf{111} (2015), no.~1, 181--204.

\bibitem[CT15]{castravettevelev}
Ana-Maria Castravet and Jenia Tevelev, \emph{$\overline{M}_{0,n}$ is not a
  {Mori} dream space}, Duke Math. J. \textbf{164} (2015), no.~8, 3851--3878.

\bibitem[CT16]{chentarasca}
Dawei Chen and Nicola Tarasca, \emph{Extremality of loci of hyperelliptic
  curves with marked \text{W}eierstrass points}, Algebra {\&} Number Theory
  \textbf{10} (2016), no.~1, 1935--1948.

\bibitem[CT19]{cavalieritarasca2017}
Renzo Cavalieri and Nicola Tarasca, \emph{Classes of weierstrass points on
  genus 2 curves},  Trans. Amer. Math. Soc. \textbf{372} (2019), no.~4, 2467--2492.

\bibitem[Fab]{faberprogram}
Carel Faber, \emph{\textsc{Maple} program for computing hodge integrals}.\\
Available at \href{http://math.stanford.edu/~vakil/programs}{\texttt{http://math.stanford.edu/\textasciitilde vakil/programs}}.

\bibitem[FP05]{faberpandharipande}
Carel Faber and Rahul Pandharipande, \emph{Relative maps and tautological
  classes}, J. Eur. Math. Soc. \textbf{7} (2005), no.~1, 13--49.

\bibitem[GK16]{gonzalezkaru}
Jos\'e~Luis Gonz\'alez and Kalle Karu, \emph{Some non-finitely generated
  \text{C}ox rings}, Compos. Math. \textbf{152} (2016), no.~5, 984--996.

\bibitem[HKL18]{hkl2016}
Jürgen Hausen, Simon Keicher, and Antonio Laface, \emph{On blowing up the
  weighted projective plane}, Mathematische Zeitschrift \textbf{290} (2018),
  1339--1358.

\bibitem[HM82]{harrismumford}
Joe Harris and David Mumford, \emph{On the \text{K}odaira dimension of the
  moduli space of curves}, Invent. Math. \textbf{67} (1982), 23--88.

\bibitem[Kee99]{keel99}
Se\'{a}n Keel, \emph{Basepoint freeness for nef and big line bundles in
  positive characteristic}, Ann. of Math. \textbf{149} (1999), no.~1, 253--286.

\bibitem[Log03]{logan2003}
Adam Logan, \emph{The \text{K}odaira dimension of moduli spaces of curves with
  marked points}, Amer. J. Math. \textbf{125} (2003), no.~1, 105--138.

\bibitem[Mul17]{mullane2017}
Scott Mullane, \emph{{On the effective cone of
  $\overline{\mathcal{M}}_{g,n}$}}, Advances in Mathematics \textbf{320}
  (2017), 500--519.

\bibitem[Opi16]{opie2016}
Morgan Opie, \emph{Extremal divisors on moduli spaces of rational curves with
  marked points}, Michigan Math. J. \textbf{65} (2016), no.~2, 251--285.

\bibitem[Rul01]{rulla01}
William~Frederick Rulla, \emph{The birational geometry of moduli space {M(3)}
  and moduli space {M(2,1)}}, Ph.D. thesis, University of Texas at Austin,
  2001.

\end{thebibliography}
\end{document}